\documentclass[3p,times]{elsarticle}

\usepackage{hyperref}

\usepackage{color}
\usepackage[whole]{bxcjkjatype}

\usepackage{graphicx}
\usepackage{amsmath}
\usepackage{amsfonts}
\usepackage{amssymb}
\usepackage{amsthm}
\usepackage{bm}
\DeclareMathAlphabet{\mathcal}{OMS}{cmsy}{m}{n}
\usepackage{hyperref}
\usepackage{mathtools}
\mathtoolsset{showonlyrefs=true}
\DeclarePairedDelimiter\paren{\lparen}{\rparen}

\newtheorem{exam}{Example}
\newtheorem{theorem}{Theorem}
\newtheorem{remark}{Remark}

\newcommand{\bbR}{\mathbb{R}}

\newcommand{\rmd}{\mathrm{d}}

\allowdisplaybreaks
\begin{document}

\begin{frontmatter}

\title{Generalization of partitioned Runge--Kutta methods for adjoint systems}
\tnotetext[mytitlenote]{Fully documented templates are available in the elsarticle package on \href{http://www.ctan.org/tex-archive/macros/latex/contrib/elsarticle}{CTAN}.}

\author[tm1,tm2]{Takeru Matsuda}
\address[tm1]{Department of Mathematical Informatics, Graduate School of Information Science and Technology, The University of Tokyo, Japan}
\address[tm2]{Mathematical Informatics Collaboration Unit, RIKEN Center for Brain Science, Japan}

\author[ym]{Yuto Miyatake}
\address[ym]{Cybermedia Center, Osaka University, Japan}

\begin{abstract}
This study computes the gradient of a function of numerical solutions of ordinary differential equations (ODEs) with respect to the initial condition.
The adjoint method computes the gradient approximately by solving the corresponding adjoint system numerically.
In this context, Sanz-Serna [SIAM Rev., 58 (2016), pp. 3--33] showed that when the initial value problem is solved by a Runge--Kutta (RK) method, the gradient can be exactly computed by applying an appropriate RK method to the adjoint system.
Focusing on the case where the initial value problem is solved by a partitioned RK (PRK) method, this paper presents a numerical method, which can be seen as a generalization of PRK methods, for the adjoint system that gives the exact gradient.

\end{abstract}


\end{frontmatter}

\section{Introduction}
\label{sec1}

We consider a system of ordinary differential equations (ODEs) of the form
\begin{align}
	\label{eq:original1}
	\frac{\rmd}{\rmd t} x(t;\theta) = f(x(t;\theta))
\end{align}
with the initial condition $x(0;\theta) = \theta \in\bbR^d$,
where $f:\bbR^d\to\bbR^d$ is assumed to be sufficiently smooth.
We denote the numerical approximation at $t_n = nh$ by $x_n(\theta) \approx x(nh;\theta)$, where $h$ is the step size.
Let $C:\bbR^d\to\bbR$ be a differentiable function.
This paper is concerned with the computation of $\nabla_\theta C(x_n(\theta))$, i.e. the gradient of $C(x_n(\theta)) $ with respect to the initial condition $\theta$.

Calculating the gradient $\nabla_\theta C(x_n(\theta))$ is a fundamental task when solving optimization problems such as
\begin{align} \label{opt_prob}
\min_\theta \sum_{n=1}^N C(x_n(\theta)).
\end{align}
A simple method of obtaining an approximation for the gradient is to compare $C(x_n(\theta))$ with the numerical approximation corresponding to a perturbed initial condition.
For example, 
$(C(x_n (\theta + \Delta e_i)) - C(x_n (\theta)))/\Delta
$
may be regarded as an approximation of the $i$th component of the gradient,
where $\Delta$ is a small scalar constant, and $e_i$ is the $i$th column of the $d$-dimensional identity matrix.
However, this approach becomes computationally expensive when the dimensionality $d$ or the number of time steps $N$ is large.
Further, the approximation accuracy could deteriorate significantly due to the discretization error.
Alternatively, the adjoint method has been widely used in various fields such as variational data assimilation in meteorology and oceanography~\cite{di86}, inversion problems in seismology~\cite{fi06},
optimal design in aerodynamics~\cite{gi00}, and neural networks~\cite{ch18}.
In this approach, the gradient is approximated by integrating the so-called adjoint system numerically. 
This approach is usually more efficient than the aforementioned approach using perturbed initial conditions; however, the approximation accuracy may still be limited when the collected discretization errors are large.

Recently, Sanz-Serna showed that the gradient $\nabla_\theta C(x_n(\theta))$ can be computed \emph{exactly} when the original system is integrated by a Runge--Kutta (RK) method~\cite{ss16}:
the intended exact gradient is obtained by applying an appropriate RK method to the adjoint system.
Note that even if the numerical integration of the original system \eqref{eq:original1} is not sufficiently accurate, a sufficiently accurate approximation or the exact computation of the gradient $\nabla_\theta C(x_n(\theta))$ is often required.
We provide two examples:
\begin{itemize}
\item The forward propagation of several deep neural networks is interpreted as a discretization of ODEs.
For example, the Residual Network (ResNet), which is commonly used in pattern recognition tasks such as image classification, can be seen as an explicit Euler discretization~\cite{cr18}.
Also, neural network architectures based on symplectic integrators for Hamiltonian systems have been developed recently, which avoid numerical instabilities~\cite{hr18}.
Since the output of such neural networks is not the exact solution of ODE but its numerical approximation, their training is formulated as an optimization problem of the form \eqref{opt_prob}.
In other words, the backpropagation algorithm~\cite{gbc16} is a special case of the adjoint method in this context.
\item Let us consider solving the optimization problem of the form \eqref{opt_prob}.
If we apply the Newton method, we need to solve a linear system having the Hessian of the objective function with respect to the initial state $\theta$ as the coefficient matrix.
When the linear system is solved by a Krylov subspace method such as the conjugate gradient (CG) method, the Hessian-vector multiplication needs to be computed.
The adjoint method can be used to approximate the Hessian-vector multiplication~\cite{wa98,wa92}.\footnote{More precisely, the Hessian-vector multiplication is approximately computed by solving the so-called second-order adjoint system numerically.}
However, this approach usually results in applying the CG method to a non-symmetric linear system even though the exact Hessian is always symmetric, and consequently, its convergence is not always guaranteed.
Symmetry can be ideally attained if the Hessian-vector multiplication is computed exactly~\cite{im19}.
We note that the need for computing the exact Hessian-vector multiplication also arises in the context of uncertainty quantification (see, e.g.~\cite{it16,th89}).
\end{itemize}

To the best of authors' knowledge, an algorithm that systematically computes the exact gradient has been developed only for the cases where the ODE system \eqref{eq:original1} is discretized by using RK methods; similar algorithms should be developed for other types of numerical methods.
We note that while it may not be so difficult to construct an algorithm for a specific ODE solver, providing a recipe for a particular class of numerical methods is useful for practitioners.
Among others, we focus on the class of partitioned RK (PRK) methods.
A straightforward conjecture would be that the exact gradient is obtained by applying an appropriate PRK method to the adjoint system.
Indeed, our previous report showed that this conjecture is true for the Störmer--Verlet method~\cite{mm19}.
However, except for some special cases, this conjecture does not hold in general. 
In this study, we shall show that the exact gradient can be computed by a certain generalization of PRK methods.

In this paper, after reviewing the approach proposed by Sanz-Serna in Section~\ref{sec2}, we show main results in Section~\ref{sec:prk}.
Several examples are provided in Section~\ref{sec:examples}, and concluding remarks are given in Section~\ref{sec:conclusion}.

\section{Preliminaries}
\label{sec2}

In this section, 
we review the adjoint method and the approach proposed by Sanz-Serna~\cite{ss16}.

\subsection{Adjoint method}
\label{sec21}
Let $\overline{x}(t)$ be the solution to the system \eqref{eq:original1} for the perturbed initial condition $\overline{x}(0) = \theta + \varepsilon$.
By linearising the system \eqref{eq:original1} at $x(t)$,
we see that as $\|\varepsilon\|\to 0$ it follows that $\overline{x}(t)	= x(t) + \delta (t) + \mathrm{o}(\|\varepsilon\|)$, where $\delta (t)$ solves the variational system
\begin{align} 
\label{vari_1}
\frac{\rmd}{\rmd t}\delta(t) = \nabla_x f(x(t)) \delta(t) .
\end{align}
The corresponding adjoint system, which is often introduced by using Lagrange multipliers, is given by
\begin{align} \label{adj_eq1}
\frac{\rmd}{\rmd t} \lambda(t) = - \nabla_x f(x(t)) ^\top \lambda (t).
\end{align}
Here, $\nabla_x f(x(t))^\top$ is the transpose of the Jacobian $\nabla_x f(x(t))$.
For any solutions to \eqref{vari_1} and \eqref{adj_eq1}, 
we see that $\frac{\rmd}{\rmd t}\lambda(t)^\top \delta(t)=0$;
thus, it follows that $\lambda (t)^\top \delta (t) = \lambda(0)^\top \delta (0)$ for all $t>0$, so in particular, 
\begin{align} \label{ld1}
    \lambda (t_N)^\top \delta (t_N) = \lambda(0)^\top \delta (0).
\end{align}
Because of the chain rule 
$
\nabla_\theta C(x(t_N;\theta)) = \nabla_\theta x(t_N;\theta)^\top
\nabla_x C(x(t_N;\theta))
$
and the fact
$
\delta (t) = (\nabla_\theta x (t;\theta)) \delta (0)
$,
we have
\begin{align} 
   \nabla_x C(x(t_N;\theta))^\top \delta (t_N) = 
   \nabla_x C(x(t_N;\theta))^\top
   (\nabla_\theta x (t_N;\theta)) \delta (0) 
   =
   \nabla_\theta C(x(t_N;\theta))^\top \delta (0). \label{ld2}
\end{align}
By comparing \eqref{ld2} with \eqref{ld1},
it is concluded that
the solution to the adjoint system \eqref{adj_eq1} with $\lambda(t_N) = \nabla_x C(x(t_N;\theta))$ satisfies $\lambda(0) = \nabla_\theta C(x(t_N;\theta))$.
In other words, solving the adjoint system \eqref{adj_eq1} backwardly with the final condition $\lambda(t_N) = \nabla_x C(x(t_N;\theta))$ gives $\lambda(0) = \nabla_\theta C(x(t_N;\theta))$.

Now, we compute the gradient $\nabla_\theta C(x_N(\theta))$.\footnote{As a particular case of $\nabla_\theta C(x_n(\theta))$, $n=1,\dots,N$, here, we consider $\nabla_\theta C(x_N(\theta))$.}
The above discussion indicates that approximating the solution to the adjoint system \eqref{adj_eq1} with the final condition $\lambda(t_N) = \nabla_x C (x_N(\theta))$ gives an approximation of $\nabla_\theta C(x_N(\theta))$ at $t=0$.
In general,
the approximation accuracy depends on the accuracy of the numerical integrators for both the original system \eqref{eq:original1} and adjoint system \eqref{adj_eq1}.

\subsection{Exact gradient calculation}
\label{sec22}

In the following, we assume that $x_n(\theta)$ is the approximation obtained by applying an RK method to the original system \eqref{eq:original1}.
In this case,
Sanz-Serna showed in~\cite{ss16} that the gradient 
$\nabla _\theta C(x_N(\theta))$ can be computed exactly (up to round-off in practical computation) by applying an appropriate RK method to the adjoint system \eqref{adj_eq1}.
Below, we briefly review the procedure.
The obvious argument $(\theta)$ will often be omitted.

Suppose that the original system \eqref{eq:original1} has been discretized by an $s$-stage RK method characterized by $\{a_{ij}\}$ and $\{b_i\}$:
\begin{align}
x_{n+1} &= x_n + h \sum_{i=1}^s b_i k_{n,i}, \\
k_{n,i} &= f\paren*{X_{n,i}}, \quad i = 1,\dots,s , \\
X_{n,i} &= x_n + h\sum_{j=1}^s a_{ij} k_{n,j}, \quad i=1,\dots,s
\end{align} 
with the initial condition $x_0 = \theta$.
The pair of $\{a_{ij}\}$ and $\{b_i\}$ will be simply denoted by $(a,b)$.
Correspondingly, the variational system \eqref{vari_1} is discretized by the same RK method
\begin{align}
\delta_{n+1} &= \delta_n + h \sum_{i=1}^s b_i d_{n,i}, \\
d_{n,i} &= \nabla_x f(X_{n,i}) \Delta_{n,i}, \quad i=1,\dots,s,\\
\Delta_{n,i} &= \delta_n + h\sum_{j=1}^s a_{ij} d_{n,j}, \quad i = 1,\dots,s
\end{align}
so that $\delta _n = (\nabla_\theta x_n (\theta)) \delta_0$.
We discretize the adjoint system \eqref{adj_eq1} with an $s$-stage RK method, which may differ from the RK method $(a,b)$,  characterized by the coefficients $(A,B)$:
\begin{align}
\begin{aligned}
\lambda_{n+1} &= \lambda_n + h \sum_{i=1}^s B_i l_{n,i}, \\
l_{n,i} &= -\nabla_x f(X_{n,i})^\top\Lambda_{n,i}, \quad i=1,\dots,s,\\
\Lambda_{n,i} &= \lambda_n + h \sum_{j=1}^s A_{ij} l_{n,j}, \quad i=1,\dots,s
\end{aligned}
\label{adj_RK_1}
\end{align}
with the final condition $\lambda_N = \nabla_x C(x_N(\theta))$.

Combining the chain rule $\nabla_\theta  C (x_N(\theta)) = \nabla_\theta x_N(\theta)^\top \nabla_x C (x_N (\theta))$ with $\delta_n = (\nabla_\theta x_n (\theta)) \delta_0$,
we see that
$\nabla_x C(x_N(\theta))^\top \delta _N 
=
\nabla_\theta C (x_N(\theta))^\top \delta_0$.
Therefore, if the approximation of the solution  to the adjoint system \eqref{adj_eq1} satisfies 
$
\lambda_n^\top \delta_n = \lambda_0^\top\delta_0
$,
in particular $\lambda_N^\top \delta_N = \lambda_0^\top\delta_0$, it is concluded that $\lambda_0 = \nabla_\theta C(x_N(\theta))$.
The theory of geometric numerical integration~\cite{hl06} tells us that if an RK method is chosen for the adjoint system such that the pair of the RK methods for the original and adjoint systems satisfies the condition for a partitioned RK (PRK) method to be canonical, the property $\lambda_N^\top \delta _N = \lambda_0^\top \delta_0$ is guaranteed and the gradient $\nabla _\theta C(x_N(\theta))$ is exactly obtained as shown in Theorem~\ref{thm:ss}.
We note that the canonical numerical methods are well-known in the context of geometric numerical integration.
For more details, we refer the reader to~\cite[Chapter VI]{hl06} (for PRK methods, see also~\cite{as93,su93}).

\begin{theorem}[\cite{ss16}]
\label{thm:ss}
Assume that the original system \eqref{eq:original1} is solved by an RK method $(a,b)$, and
the coefficients $(A,B)$ of the RK method for the adjoint system \eqref{adj_eq1} satisfy
\begin{align}
& b_i = B_i, \quad i = 1,\dots,s , \\
& b_i A_{ij} + B_j a_{ji} = b_i B_j, \quad i,j=1,\dots,s.
\end{align}
Then, by solving the adjoint system with the final condition $\lambda_N (=\lambda(t_N
))= \nabla_x C(x_N(\theta))$, we obtain the exact gradient $\lambda_0 = \nabla _\theta C(x_N(\theta))$.
\end{theorem}

\begin{remark}
The conditions in Theorem~\ref{thm:ss} indicate that 
$ A_{ij} = b_j - (b_j/b_i)a_{ji}$,
which makes sense only when every weight $b_i$ is nonzero.
However, for some RK methods
one or more of the weights $b_i$ vanish.
In such cases, the above conditions cannot be used to find an appropriate RK method for the adjoint system.
We refer the reader to Appendix in~\cite{ss16} for a workaround.
\end{remark}

\section{Partitioned Runge--Kutta methods}
\label{sec:prk}

We now restrict our attention to the system of ODEs in the partitioned form
\begin{align}
\label{eq:p_original1}
\frac{\rmd}{\rmd t}
\begin{bmatrix}
x^{[1]} \\ x^{[2]} 
\end{bmatrix}
=
\begin{bmatrix}
f^{[1]} \paren*{ x^{[1]}, x^{[2]}} \\
f^{[2]} \paren*{ x^{[1]}, x^{[2]}}
\end{bmatrix},
\quad
\begin{bmatrix}
x^{[1]} (0) \\ x^{[2]}  (0)
\end{bmatrix}
=
\begin{bmatrix}
\theta^{[1]} \\ \theta^{[2]}
\end{bmatrix}.
\end{align}
Suppose that this system has been discretized by a PRK method characterized by the pairs $(a^{[1]},b^{[1]})$ and $(a^{[2]},b^{[2]})$:
\begin{align}
\begin{bmatrix} \displaystyle
x_{n+1}^{[1]}  \\ x_{n+1}^{[2]}
\end{bmatrix}
&= 
\begin{bmatrix}
x_n^{[1]} \\ x_n^{[2]}
\end{bmatrix}
 + h \sum_{i=1}^s
\begin{bmatrix} 
\displaystyle
  b_i^{[1]} k_{n,i}^{[1]}\\
\displaystyle
  b_i^{[2]} k_{n,i}^{[2]}
\end{bmatrix},\\
\begin{bmatrix}
k_{n,i}^{[1]}  \\ k_{n,i}^{[2]} 
\end{bmatrix}
&=
\begin{bmatrix}
f^{[1]} \paren*{X_{n,i}^{[1]},X_{n,i}^{[2]}} \\
 f^{[2]} \paren*{X_{n,i}^{[1]},X_{n,i}^{[2]}}
\end{bmatrix}, \quad i = 1,\dots,s,\\
\begin{bmatrix}
X_{n,i}^{[1]} \\
X_{n,i}^{[2]} 
\end{bmatrix}
&=
\begin{bmatrix}
x_n^{[1]} \\
x_n^{[2]}
\end{bmatrix}
+
h \sum_{j=1}^s
\begin{bmatrix} 
\displaystyle
   a_{ij}^{[1]} k_{n,j}^{[1]}\\
\displaystyle
  a_{ij}^{[2]} k_{n,j}^{[2]}
\end{bmatrix}, \quad i = 1,\dots,s.
\end{align}
For clarity, every component of $b^{[1]}$ and $b^{[2]}$ is assumed to be nonzero.
We discretize the variational system expressed as
\begin{align}
\label{eq:p_variational1}
    \frac{\rmd}{\rmd t}
    \begin{bmatrix}
    \delta^{[1]} \\ \delta^{[2]}
    \end{bmatrix}
    =
    \begin{bmatrix}
    \nabla_{x^{[1]}} f^{[1]} & \nabla_{x^{[2]}} f^{[1]} \\
    \nabla_{x^{[1]}} f^{[2]} & \nabla_{x^{[2]}} f^{[2]}
    \end{bmatrix}
    \begin{bmatrix}
    \delta^{[1]} \\ \delta^{[2]}
    \end{bmatrix}
\end{align}
by the same PRK method
\begin{subequations}
\begin{align}
    \begin{bmatrix}
    \delta_{n+1}^{[1]} \\ \delta_{n+1}^{[2]}
    \end{bmatrix}
    &=
    \begin{bmatrix}
    \delta_n^{[1]} \\ \delta_n^{[2]}
    \end{bmatrix}
    +
    h \sum_{i=1}^s 
    \begin{bmatrix}
    b_i ^{[1]} m_{n,i}^{[1]} \\
    b_i ^{[2]} m_{n,i}^{[2]}
    \end{bmatrix},\label{PRKv1} \\
    \begin{bmatrix}
    m_{n,i}^{[1]}  \\  m_{n,i}^{[2]} 
    \end{bmatrix}
    &=
    \begin{bmatrix}
    (\nabla_{x^{[1]}}f^{[1]}_{n,i}) D_{n,i}^{[1]}
    +
    (\nabla_{x^{[2]}}f^{[1]}_{n,i}) D_{n,i}^{[2]} \\
    (\nabla_{x^{[1]}}f^{[2]}_{n,i}) D_{n,i}^{[1]}
    +
    (\nabla_{x^{[2]}}f^{[2]}_{n,i}) D_{n,i}^{[2]}
    \end{bmatrix}, \quad i = 1,\dots, s, \label{PRKv2}\\
    \begin{bmatrix}
    D_{n,i}^{[1]} \\
    D_{n,i}^{[2]}
    \end{bmatrix}
    &=
    \begin{bmatrix}
    \delta_n^{[1]} \\ \delta_n^{[2]}
    \end{bmatrix}
    +
    h \sum_{j=1}^s 
    \begin{bmatrix}
    a_{ij} ^{[1]} m_{n,j}^{[1]} \\
    a_{ij} ^{[2]} m_{n,j}^{[2]}
    \end{bmatrix}, \quad i =1,\dots,s. \label{PRKv3}
\end{align}
\end{subequations}

The corresponding adjoint system of \eqref{eq:p_original1} is written as
\begin{align}
\label{p_adj_eq1}
\frac{\rmd}{\rmd t}
\begin{bmatrix}
\lambda^{[1]} \\
\lambda^{[2]}
\end{bmatrix}
=
-
\begin{bmatrix}
\nabla_{x^{[1]}} f^{[1]} (x^{[1]},x^{[2]}) &
\nabla_{x^{[2]}} f^{[1]} (x^{[1]},x^{[2]}) \\
\nabla_{x^{[1]}} f^{[2]} (x^{[1]},x^{[2]}) &
\nabla_{x^{[2]}} f^{[2]} (x^{[1]},x^{[2]})
\end{bmatrix}^\top
\begin{bmatrix}
\lambda^{[1]} \\
\lambda^{[2]}
\end{bmatrix},
\end{align}
for which the following final condition
\begin{align}
\label{p_fsc}
\begin{bmatrix}
\lambda^{[1]} (t_N) \\
\lambda^{[2]} (t_N)
\end{bmatrix}
=
\begin{bmatrix}
\nabla_{x^{[1]}} C (x_N^{[1]},x_N^{[2]}) \\
\nabla_{x^{[2]}} C (x_N^{[1]},x_N^{[2]})
\end{bmatrix}
\end{align}
is imposed.
We consider discretizing the adjoint system by the following formula
\begin{subequations}
\begin{align}
\begin{bmatrix}
\lambda^{[1]}_{n+1} \\ 
\lambda^{[2]}_{n+1}
\end{bmatrix}
&=
\begin{bmatrix}
\lambda^{[1]}_{n} \\ 
\lambda^{[2]}_{n}
\end{bmatrix}
-h
\sum_{i=1}^s
\begin{bmatrix}
B_i^{[11]} l_{n,i}^{[11]}
+
B_i^{[12]} l_{n,i}^{[12]} \\
B_i^{[21]} l_{n,i}^{[21]}
+
B_i^{[22]} l_{n,i}^{[22]}
\end{bmatrix}, \label{GPRK1}\\
\begin{bmatrix}
l_{n,i}^{[11]} & l_{n,i}^{[12]}  \\
l_{n,i}^{[21]}  & l_{n,i}^{[22]}
\end{bmatrix}
&=
\begin{bmatrix}
(\nabla_{x^{[1]}} f_{n,i}^{[1]})^\top \Lambda_{n,i}^{[1]} &
(\nabla_{x^{[1]}} f_{n,i}^{[2]})^\top \Lambda_{n,i}^{[2]} \\
(\nabla_{x^{[2]}} f_{n,i}^{[1]})^\top \Lambda_{n,i}^{[1]} &
(\nabla_{x^{[2]}} f_{n,i}^{[2]})^\top \Lambda_{n,i}^{[2]}
\end{bmatrix},\quad i = 1,\dots,s, \label{GPRK2}\\
\begin{bmatrix}
\Lambda^{[1]}_{n,i} \\ 
\Lambda^{[2]}_{n,i}
\end{bmatrix}
&=
\begin{bmatrix}
\lambda^{[1]}_{n} \\ 
\lambda^{[2]}_{n}
\end{bmatrix}
-h
\sum_{j=1}^s
\begin{bmatrix}
A_{ij}^{[11]} l_{n,j}^{[11]}
+
A_{ij}^{[12]} l_{n,j}^{[12]} \\
A_{ij}^{[21]} l_{n,j}^{[21]}
+
A_{ij}^{[22]} l_{n,j}^{[22]}
\end{bmatrix}, \quad i = 1,\dots,s.\label{GPRK3}
\end{align}
\end{subequations}

We note that the above formula (\eqref{GPRK1}, \eqref{GPRK2}, \eqref{GPRK3}) does not necessarily belong to the class of PRK methods, and thus refer to the formula as a generalized PRK (GPRK) method.\footnote{This kind of generalization is only applicable for partitioned systems in which the vector field is decomposed into two parts.}
The formula falls into a PRK method only when the coefficients satisfy
\begin{align}\label{reduction}
    B^{[11]} = B^{[12]}, \quad B^{[21]} = B^{[22]}, \quad
    A^{[11]} = A^{[12]}, \quad A^{[21]} = A^{[22]}.
\end{align}

The following theorem shows that the gradient can be computed exactly by an appropriate GPRK method.

\begin{theorem}
\label{thm:new}
Assume that the original system \eqref{eq:p_original1} is solved by a PRK method characterized by $(a^{[1]},b^{[1]})$ and $(a^{[2]},b^{[2]})$, and
the coefficients of the GPRK method \eqref{GPRK1}-\eqref{GPRK3} for the adjoint system \eqref{p_adj_eq1} satisfy
\begin{align}
& b_i^{[1]} = B_i^{[11]}=B_i^{[21]}, \quad i = 1,\dots,s , \\
& b_i^{[2]} = B_i^{[12]}=B_i^{[22]}, \quad i = 1,\dots,s , \\
& b_i^{[1]} A_{ij}^{[11]} + B_j^{[11]} a_{ji}^{[1]} = b_i^{[1]} B_j^{[11]}, \quad i,j=1,\dots,s, \\
& b_i^{[1]} A_{ij}^{[12]} + B_j^{[12]} a_{ji}^{[1]} = b_i^{[1]} B_j^{[12]}, \quad i,j=1,\dots,s, \\
& b_i^{[2]} A_{ij}^{[21]} + B_j^{[21]} a_{ji}^{[2]} = b_i^{[2]} B_j^{[21]}, \quad i,j=1,\dots,s, \\
& b_i^{[2]} A_{ij}^{[22]} + B_j^{[22]} a_{ji}^{[2]} = b_i^{[2]} B_j^{[22]}, \quad i,j=1,\dots,s.
\end{align}
Then, by solving the adjoint system with the final condition \eqref{p_fsc}, we obtain the exact gradient $\lambda_0^{[1]} = \nabla _{\theta^{[1]}} C(x_N^{[1]}(\theta),x_N^{[2]}(\theta))$ and $\lambda_0^{[2]} = \nabla _{\theta^{[2]}} C(x_N^{[1]}(\theta),x_N^{[2]}(\theta))$.
\end{theorem}

\begin{proof}
It suffices to show that $\lambda_{n+1}^\top \delta_{n+1} = \lambda_n ^\top \delta_n$.
We see from \eqref{PRKv1} and \eqref{GPRK1} that
\begin{align}
    &(\lambda_{n+1}^{[1]})^\top \delta_{n+1}^{[1]}
    +
    (\lambda_{n+1}^{[2]})^\top \delta_{n+1}^{[2]} \\
    & \quad =
    \paren*{\lambda_n^{[1]}-h\sum_{j=1}^s (B_j^{[11]} l_{n,j}^{[11]} + B_j^{[12]} l_{n,j}^{[12]})}^\top
    \paren*{\delta_n^{[1]} + h \sum_{i=1}^s b_i^{[1]} m_{n,i}^{[1]}} \\
    & \quad \phantom{=}
    +
    \paren*{\lambda_n^{[2]}-h\sum_{j=1}^s (B_j^{[21]} l_{n,j}^{[21]} + B_j^{[22]} l_{n,j}^{[22]})}^\top
    \paren*{\delta_n^{[2]} + h \sum_{i=1}^s b_i^{[2]} m_{n,i}^{[2]}} \\
    & \quad =
    (\lambda_n^{[1]})^\top \delta_n^{[1]}
    +
    (\lambda_n^{[2]})^\top \delta_n^{[2]} \\
    & \quad \phantom{=}
    +
    h\sum_{i=1}^s b_i^{[1]} (\lambda_n^{[1]})^\top m_{n,i}^{[1]}
    -
    h\sum_{j=1}^s (B_j^{[11]} l_{n,j}^{[11]} + B_j^{[12]} l_{n,j}^{[12]} ) ^\top \delta_n^{[1]} 
    -
    h^2\sum_{i,j=1}^s b_i^{[1]} (B_j^{[11]} l_{n,j}^{[11]} + B_j^{[12]} l_{n,j}^{[12]} ) ^\top m_{n,i}^{[1]} \\
    & \quad \phantom{=}
    +
    h\sum_{i=1}^s b_i^{[2]} (\lambda_n^{[2]})^\top m_{n,i}^{[2]}
    -
    h\sum_{j=1}^s (B_j^{[21]} l_{n,j}^{[21]} + B_j^{[22]} l_{n,j}^{[22]} ) ^\top \delta_n^{[2]} 
    -
    h^2\sum_{i,j=1}^s b_i^{[2]} (B_j^{[21]} l_{n,j}^{[21]} + B_j^{[22]} l_{n,j}^{[22]} ) ^\top m_{n,i}^{[2]}.  \label{lamdel_p}
\end{align}
By substituting \eqref{PRKv2}, \eqref{PRKv3}, \eqref{GPRK2} and \eqref{GPRK3} to \eqref{lamdel_p}, we have
\begin{align}
    &(\lambda_{n+1}^{[1]})^\top \delta_{n+1}^{[1]}
    +
    (\lambda_{n+1}^{[2]})^\top \delta_{n+1}^{[2]} \\
    & \quad =   
    (\lambda_n^{[1]})^\top \delta_n^{[1]}
    +
    (\lambda_n^{[2]})^\top \delta_n^{[2]} \\
    & \quad \phantom{=}
    +
    h \sum_{i=1}^s (b_i^{[1]} - B_i^{[11]}) (\lambda_{n,i}^{[1]})^\top (\nabla_{x^{[1]}} f_{n,i}^{[1]}) D_{n,i}^{[1]}
    +
    h \sum_{i=1}^s (b_i^{[1]} - B_i^{[21]}) (\lambda_{n,i}^{[1]})^\top (\nabla_{x^{[2]}} f_{n,i}^{[1]}) D_{n,i}^{[2]} \\
    & \quad \phantom{=}
    +
    h \sum_{i=1}^s (b_i^{[2]} - B_i^{[12]}) (\lambda_{n,i}^{[2]})^\top (\nabla_{x^{[1]}} f_{n,i}^{[2]}) D_{n,i}^{[1]}
    +
    h \sum_{i=1}^s (b_i^{[2]} - B_i^{[22]}) (\lambda_{n,i}^{[2]})^\top (\nabla_{x^{[2]}} f_{n,i}^{[2]}) D_{n,i}^{[2]} \\
    & \quad \phantom{=}
    +
    h^2 \sum_{i,j=1}^s
    (b_i^{[1]} A_{ij}^{[11]} + B_j^{[11]} a_{ji}^{[1]} - b_i^{[1]} B_j^{[11]})  (l_{n,j}^{[11]})^\top m_{n,i}^{[1]}
    +
    h^2 \sum_{i,j=1}^s
    (b_i^{[1]} A_{ij}^{[12]} + B_j^{[12]} a_{ji}^{[1]} - b_i^{[1]} B_j^{[12]})  (l_{n,j}^{[12]})^\top m_{n,i}^{[1]} \\
    & \quad \phantom{=}
    +
    h^2 \sum_{i,j=1}^s
    (b_i^{[2]} A_{ij}^{[21]} + B_j^{[21]} a_{ji}^{[2]} - b_i^{[2]} B_j^{[21]})  (l_{n,j}^{[21]})^\top m_{n,i}^{[2]} 
    +
    h^2 \sum_{i,j=1}^s
    (b_i^{[2]} A_{ij}^{[22]} + B_j^{[22]} a_{ji}^{[2]} - b_i^{[2]} B_j^{[22]})  (l_{n,j}^{[22]})^\top m_{n,i}^{[2]}.
\end{align}
Hence, if the conditions in Theorem~\ref{thm:new} are satisfied, it follows that $\lambda_{n+1}^\top \delta_{n+1} = \lambda_n ^\top \delta_n$.
\end{proof}

\begin{remark} \label{rem:reduction}
We note that if the coefficients of the PRK method applied to the system \eqref{eq:p_original1} satisfy $b^{[1]} = b^{[2]}$, 
the GPRK method, which is uniquely determined from the conditions in Theorem~\ref{thm:new}, satisfy \eqref{reduction}. Therefore, in this case, the GPRK method falls into the category of PRK methods.
\end{remark}

\section{Examples}
\label{sec:examples}
We consider several classes of PRK methods and show corresponding GPRK methods that are uniquely determined from the conditions in Theorem~\ref{thm:new}.

\begin{exam}[Symplectic PRK methods for Hamiltonian systems]

For Hamiltonian systems,
the coefficients of most symplectic PRK methods, such as the Störmer--Verlet method and the 3-stage Lobatto IIIA--IIIB pair, satisfy
\begin{align}
& b_i^{[1]} = b_i^{[2]}, \quad i=1,\dots,s,  \label{symp_cond1}\\
& b_i^{[1]} a_{ij}^{[2]} + b_j^{[2]} a_{ji}^{[1]} 
= b_i^{[1]} b_j^{[2]}, \quad, i,j=1,\dots,s. \label{symp_cond2}
\end{align}
In this case, the coefficients of the GPRK method satisfying the conditions in Theorem~\ref{thm:new} are explicitly given by
\begin{align}
 B^{[11]} = B^{[21]} = B^{[12]} = B^{[22]} = b, \quad
 A^{[11]} = A^{[12]} =a^{[2]}, \quad A^{[21]} = A^{[22]} = a^{[1]},
\end{align}
where $b:= b^{[1]} = b^{[2]}$.
The GPRK method reduces to a PRK method.
We note that the GPRK method for the Störmer--Verlet method is consistent with the one which was presented in the context of inverse problem for ODEs~\cite{mm19}, and essentially the same algorithm for a certain vector field $f$ can also be found in the context of deep neural networks~\cite{hr18}. 
\end{exam}

\begin{exam}[Symplectic PRK methods for separable Hamiltonian systems]

For separable Hamiltonian systems, a PRK method is symplectic even if the first condition \eqref{symp_cond1} is violated.
We consider the case where $b^{[1]} \neq b^{[2]}$.
In this case, the GPRK method satisfying the conditions in Theorem~\ref{thm:new} does not reduces to a PRK method.
However, when the system \eqref{eq:p_variational1} is separable, i.e. $f^{[1]} = f^{[1]} ( x^{[2]})$ and $f^{[2]} = f^{[2]} (x^{[1]})$, we see from \eqref{GPRK2} that $l_{n,i}^{[11]} = l_{n,i}^{[22]} = 0$. Thus the coefficients $B^{[11]}$, $B^{[22]}$, $A^{[11]}$ and $A^{[22]}$ are no longer needed, and the GPRK method reduces to a PRK method.
The remaining coefficients are explicitly given by
\begin{align}
B^{[21]} = b^{[1]}, \quad B^{[12]} = b^{[2]}, \quad
A^{[12]}= a^{[2]}, \quad A^{[21]} = a^{[1]}. 
\end{align}

\end{exam}

\begin{exam}[Spatially partitioned embedded RK methods]

Spatially partitioned embedded RK methods are sometimes applied to a system of ODEs arising from discretizing partial differential equations~\cite{km13}.
The methods belong to a class of PRK methods with $a^{[1]} = a^{[2]} (=a)$ and $b^{[1]} \neq b^{[2]}$.
In this case, the GPRK method satisfying the conditions in Theorem~\ref{thm:new} is no longer a PRK method.
The coefficients are given by
\begin{align}
& B^{[11]} = B^{[21]} =b^{[1]} , \quad
B^{[12]} = B^{[22]} =b^{[2]} , \\
& A_{ij}^{[11]} = b_j^{[1]} - \frac{b_j^{[1]}}{b_i^{[1]}} a_{ji} ,\quad A_{ij}^{[12]} = b_j^{[2]} - \frac{b_j^{[2]}}{b_i^{[1]}} a_{ji} ,\quad A_{ij}^{[21]} = b_j^{[1]} - \frac{b_j^{[1]}}{b_i^{[2]}} a_{ji} ,\quad A_{ij}^{[22]} = b_j^{[2]} - \frac{b_j^{[2]}}{b_i^{[2]}} a_{ji} ,\quad i,j=1,\dots,s.
\end{align}
\end{exam}

\section{Concluding remarks}
\label{sec:conclusion}

In this paper, we have shown that the gradient of a function of numerical solutions to ODEs with respect to the initial condition can be systematically and efficiently computed when the ODE is solved by a partitioned Runge--Kutta method.
The key idea is to apply a certain generalization of partitioned Runge--Kutta methods to the corresponding adjoint system.
The proposed method is applicable to, for example,  estimation of initial condition or underlying system parameters of ODE models and training of deep neural networks.

It is an interesting problem to construct a similar recipe for other types of ODE solvers such as linear multistep methods and exponential integrators. 
We are now working on this issue.

\section*{Acknowledgement}
This work was supported in part by JST ACT-I Grant Number JPMJPR18US, and JSPS Grants-in-Aid for Early-Career Scientists Grant Numbers 16K17550 and 19K20220.

\bibliography{references}

\begin{thebibliography}{18}
\expandafter\ifx\csname natexlab\endcsname\relax\def\natexlab#1{#1}\fi
\providecommand{\bibinfo}[2]{#2}
\ifx\xfnm\relax \def\xfnm[#1]{\unskip,\space#1}\fi
\bibitem[{Abia and Sanz-Serna(1993)}]{as93}
\bibinfo{author}{L.~Abia}, \bibinfo{author}{J.M. Sanz-Serna},
  \bibinfo{title}{Partitioned {R}unge--{K}utta methods for separable
  {H}amiltonian problems}, \bibinfo{journal}{Math. Comp.} \bibinfo{volume}{60}
  (\bibinfo{year}{1993}) \bibinfo{pages}{617--634}.
\bibitem[{Chen et~al.(2018{\natexlab{a}})Chen, Rubanova, Bettencourt and
  Duvenaud}]{ch18}
\bibinfo{author}{R.T. Chen}, \bibinfo{author}{Y.~Rubanova},
  \bibinfo{author}{J.~Bettencourt}, \bibinfo{author}{D.~Duvenaud},
  \bibinfo{title}{Neural ordinary differential equations},
  \bibinfo{journal}{arXiv:1806.07366}  (\bibinfo{year}{2018}{\natexlab{a}}).
\bibitem[{Chen et~al.(2018{\natexlab{b}})Chen, Rubanova, Bettencourt and
  Duvenaud}]{cr18}
\bibinfo{author}{R.T.Q. Chen}, \bibinfo{author}{Y.~Rubanova},
  \bibinfo{author}{J.~Bettencourt}, \bibinfo{author}{D.K. Duvenaud},
  \bibinfo{title}{Neural ordinary differential equations}, in:
  \bibinfo{booktitle}{Advances in Neural Information Processing Systems 31},
  \bibinfo{year}{2018}{\natexlab{b}}, pp. \bibinfo{pages}{6571--6583}.
\bibitem[{Dimet and Talagrand(1986)}]{di86}
\bibinfo{author}{F.X.L. Dimet}, \bibinfo{author}{O.~Talagrand},
  \bibinfo{title}{Variational algorithms for analysis and assimilation of
  meteorological observations: theoretical aspects}, \bibinfo{journal}{Tellus
  A} \bibinfo{volume}{38A} (\bibinfo{year}{1986}) \bibinfo{pages}{97--110}.
\bibitem[{Fichtner et~al.(2006)Fichtner, Bunge and Igel}]{fi06}
\bibinfo{author}{A.~Fichtner}, \bibinfo{author}{H.P. Bunge},
  \bibinfo{author}{H.~Igel}, \bibinfo{title}{The adjoint method in seismology:
  I. theory}, \bibinfo{journal}{Phys. Earth Planet. In.} \bibinfo{volume}{157}
  (\bibinfo{year}{2006}) \bibinfo{pages}{86--104}.
\bibitem[{Giles(2000)}]{gi00}
\bibinfo{author}{N.A. Giles, Michael B.and~Pierce}, \bibinfo{title}{An
  introduction to the adjoint approach to design}, \bibinfo{journal}{Flow,
  Turbul. Combust.} \bibinfo{volume}{65} (\bibinfo{year}{2000})
  \bibinfo{pages}{393--415}.
\bibitem[{Goodfellow et~al.(2016)Goodfellow, Bengio and Courville}]{gbc16}
\bibinfo{author}{I.~Goodfellow}, \bibinfo{author}{Y.~Bengio},
  \bibinfo{author}{A.~Courville}, \bibinfo{title}{Deep Learning},
  \bibinfo{publisher}{MIT Press}, \bibinfo{year}{2016}.
  \bibinfo{note}{\url{http://www.deeplearningbook.org}}.
\bibitem[{Haber and Ruthotto(2018)}]{hr18}
\bibinfo{author}{E.~Haber}, \bibinfo{author}{L.~Ruthotto},
  \bibinfo{title}{Stable architectures for deep neural networks},
  \bibinfo{journal}{Inverse Problems} \bibinfo{volume}{34}
  (\bibinfo{year}{2018}) \bibinfo{pages}{014004}.
\bibitem[{Hairer et~al.(2006)Hairer, Lubich and Wanner}]{hl06}
\bibinfo{author}{E.~Hairer}, \bibinfo{author}{C.~Lubich},
  \bibinfo{author}{G.~Wanner}, \bibinfo{title}{Geometric Numerical Integration:
  Structure-Preserving Algorithms for Ordinary Differential quations},
  \bibinfo{publisher}{Springer-Verlag, Berlin}, \bibinfo{edition}{second}
  edition, \bibinfo{year}{2006}.
\bibitem[{Ito et~al.(2019)Ito, Matsuda and Miyatake}]{im19}
\bibinfo{author}{S.~Ito}, \bibinfo{author}{T.~Matsuda},
  \bibinfo{author}{Y.~Miyatake}, \bibinfo{title}{Adjoint-based exact
  hessian-vector multiplication using symplectic runge--kutta methods},
  \bibinfo{journal}{arXiv:1910.06524}  (\bibinfo{year}{2019}).
\bibitem[{Ito et~al.(2016)Ito, Nagao, Yamanaka, Tsukada, Koyama, Kano and
  Inoue}]{it16}
\bibinfo{author}{S.~Ito}, \bibinfo{author}{H.~Nagao},
  \bibinfo{author}{A.~Yamanaka}, \bibinfo{author}{Y.~Tsukada},
  \bibinfo{author}{T.~Koyama}, \bibinfo{author}{M.~Kano},
  \bibinfo{author}{J.~Inoue}, \bibinfo{title}{Data assimilation for massive
  autonomous systems based on a second-order adjoint method},
  \bibinfo{journal}{Phys. Rev. E} \bibinfo{volume}{94} (\bibinfo{year}{2016})
  \bibinfo{pages}{043307}.
\bibitem[{Ketcheson et~al.(2013)Ketcheson, MacDonald and Ruuth}]{km13}
\bibinfo{author}{D.I. Ketcheson}, \bibinfo{author}{C.B. MacDonald},
  \bibinfo{author}{S.J. Ruuth}, \bibinfo{title}{Spatially partitioned embedded
  {R}unge-{K}utta methods}, \bibinfo{journal}{SIAM J. Numer. Anal.}
  \bibinfo{volume}{51} (\bibinfo{year}{2013}) \bibinfo{pages}{2887--2910}.
\bibitem[{Matsuda and Miyatake(2019)}]{mm19}
\bibinfo{author}{T.~Matsuda}, \bibinfo{author}{Y.~Miyatake},
  \bibinfo{title}{Estimation of ordinary differential equation models with
  discretization error quantification}, \bibinfo{journal}{arXiv:1907.10565}
  (\bibinfo{year}{2019}).
\bibitem[{Sanz-Serna(2016)}]{ss16}
\bibinfo{author}{J.M. Sanz-Serna}, \bibinfo{title}{Symplectic {R}unge--{K}utta
  schemes for adjoint equations, automatic differentiation, optimal control,
  and more}, \bibinfo{journal}{SIAM Rev.} \bibinfo{volume}{58}
  (\bibinfo{year}{2016}) \bibinfo{pages}{3--33}.
\bibitem[{Sun(1993)}]{su93}
\bibinfo{author}{G.~Sun}, \bibinfo{title}{Symplectic partitioned
  {R}unge--{K}utta methods}, \bibinfo{journal}{J. Comput. Math.}
  \bibinfo{volume}{11} (\bibinfo{year}{1993}) \bibinfo{pages}{365--372}.
\bibitem[{Thacker(1989)}]{th89}
\bibinfo{author}{W.C. Thacker}, \bibinfo{title}{The role of the hessian matrix
  in fitting models to measurements}, \bibinfo{journal}{J. Geophys. Res.
  Oceans} \bibinfo{volume}{94} (\bibinfo{year}{1989})
  \bibinfo{pages}{6177--6196}.
\bibitem[{Wang et~al.(1998)Wang, Droegemeier and White}]{wa98}
\bibinfo{author}{Z.~Wang}, \bibinfo{author}{K.~Droegemeier},
  \bibinfo{author}{L.~White}, \bibinfo{title}{The adjoint {N}ewton algorithm
  for large-scale unconstrained optimization in meteorology applications},
  \bibinfo{journal}{Comput. Optim. and Appl.} \bibinfo{volume}{10}
  (\bibinfo{year}{1998}) \bibinfo{pages}{283--320}.
\bibitem[{Wang et~al.(1992)Wang, Navon, Dimet and Zou}]{wa92}
\bibinfo{author}{Z.~Wang}, \bibinfo{author}{I.M. Navon},
  \bibinfo{author}{F.X.L. Dimet}, \bibinfo{author}{X.~Zou}, \bibinfo{title}{The
  second order adjoint analysis: Theory and applications},
  \bibinfo{journal}{Meteor. Atmos. Phys.} \bibinfo{volume}{50}
  (\bibinfo{year}{1992}) \bibinfo{pages}{3--20}.

\end{thebibliography}

\end{document}